\documentclass[12pt]{amsart}

\usepackage{amsmath,amssymb, mathtools} % basic maths packages
\usepackage{amsthm} % for theorem environments
\usepackage{dsfont} % for mathematical fonts like mathbb{R}
\usepackage{tensor} % for tensor indices

\usepackage[pdftex]{graphicx} % for including pictures with includegraphics
\usepackage{color}
\usepackage[colorlinks,citecolor=blue,linkcolor=blue,urlcolor=blue]{hyperref} % for references in the text with colors
\usepackage{cleveref} % clever references (recognizes type like Theorem, Lemma etc)

\usepackage[utf8]{inputenc} % to translate accents to intern latex code
\usepackage[english]{babel} % for english language

\usepackage{geometry} % for size, margins, ...
\usepackage{tikz} % for pictures with tikz
\usepackage{cite} %bibtex
\usepackage{enumitem} %indent of the tables in Appendix

\newtheorem{thm}{Theorem}
\newtheorem{conj}[thm]{Conjecture}
\newtheorem{lemma}[thm]{Lemma}

\newtheorem{definition}[thm]{Definition}
\newtheorem{prop}[thm]{Proposition}
\newtheorem{example}[thm]{Example}
\newtheorem{Remark}[thm]{Remark}
\newtheorem{oq}[thm]{Open question}
\newtheorem{fakethm}{Theorem}

\newtheorem{fakeconj}{Conjecture}

\numberwithin{thm}{section} % theorem labeling within sections
\numberwithin{equation}{section} % equation labeling within sections
\numberwithin{figure}{section} % figure labeling within sections

\newcommand*{\R}{\mathbb{R}}
\newcommand*{\C}{\mathbb{C}}

\renewcommand*{\S}{\Sigma}

\newcommand*{\mc}{\mathcal}
\newcommand*{\bb}{\mathbb}
\newcommand*{\del}{\partial}
\newcommand*{\delbar}{\bar\partial}
\newcommand*{\cotang}{T^*\mc{T}^n}
\newcommand*{\T}{\mc{T}}

\DeclareMathOperator{\Rep}{Rep}
\DeclareMathOperator{\id}{id}

\DeclareMathOperator{\Hom}{Hom}

\DeclareMathOperator{\PSL}{PSL}
\DeclareMathOperator{\SL}{SL}

\DeclareMathOperator{\Hilb}{Hilb}
\DeclareMathOperator{\Symp}{Symp}

\author{Alexander Thomas}
\title{Differential Operators on Surfaces and Rational WKB Method}
\address{Max-Planck-Institute for Mathematics, Vivatsgasse 7, 53111 Bonn, Germany}
\email{athomas@mpim-bonn.mpg.de}

\begin{document}

\begin{abstract}
In this paper, we give a simple and geometric, but formal, description of an open subset of the character variety of surface groups into $\SL_n(\C)$. The main ingredient is a modified version of the WKB method, which we call rational WKB method. 

The geometric interpretation uses higher complex structures introduced by Vladimir Fock and the author. More precisely, the character variety is parametrized by the cotangent bundle of the moduli space of higher complex structures. This generalizes the well-known description of the moduli space of flat $\SL_2(\C)$-connections by the cotangent bundle of Teichmüller space.
\end{abstract}

\maketitle

\tableofcontents

\section{Introduction}

The motivating question for this paper is \textit{``What is a linear differential operator on a surface?''} To be more specific: what kind of global object on a surface can be seen as a generalization of differential operators and how can we parametrize these objects in a simple way?

In dimension 1, a linear differential operator can be put into the form $$D = d^n-t_1d^{n-1}-t_2d^{n-2}-...-t_n$$
where we have put $d=\frac{d}{dx}$.
Such a differential operator of order $n$ is equivalent to a matrix-valued differential operator of order 1. In terms of differential equations, the equation $$(d^n-t_1d^{n-1}-t_2d^{n-2}-...-t_n)\psi=0$$
is equivalent to $$(d-A)\Psi=0$$ where
$$A = \begin{pmatrix} &1&&\\ && \ddots &\\ &&&1\\ t_n & t_{n-1} &\cdots & t_1\end{pmatrix} \;\text{ and }\; \Psi=\begin{pmatrix}\psi\\ d\psi \\ \vdots \\ d^{n-1}\psi\end{pmatrix}.$$
 
In this paper, we will use the slightly different form $\Psi^T(d-A^T)=0$ where the differential operator acts from the right. This allows to use the transpose $A^T$ in which the parameters $t_i$ are in a column, allowing a better geometric interpretation (via Equation \eqref{eqqq1} below).

The global nature of an expression of the local form $d+A$, where $A$ is a matrix-valued 1-form, is a \emph{connection} on some bundle. Thus, a differential operator on a manifold should be a connection. We also require that a linear differential operator of order $n$ should have $n$ independent solutions. This implies that the connection has to be \textit{flat} (i.e. with vanishing curvature). This is automatically true in dimension 1. Here we will work in dimension 2, over a surface $S$.

%In the special case where the manifold is a surface $\S$, the space of all $G$-connections $\mathcal{A}(\S,G)$ has a symplectic structure. The symplectic form is given by $$\omega = \int_\S \tr \delta A\wedge \delta A.$$
%In addition, the space of all connections has natural symmetries given by the action of the gauge group.
%The famous theorem of Atiyah--Bott states that the \emph{gauge action is hamiltonian with moment map given by the curvature} (see \cite{atiyah1983yang} at the end of Section 9 and Donaldson's viewpoint in \cite[Section 1.1]{donaldson2003moment}). Hence the moduli space of flat connections, i.e. the space of flat connections modulo gauge, is symplectic.

A flat connection allows to define parallel transport. This gives a representation of the fundamental group $\pi_1(\S)$ into the structure group $G$ via its monodromy. We denote by $$\Rep(\pi_1(\S),G) = \Hom(\pi_1(S),G)/G$$ the representation variety, or \textbf{character variety}.
Any point in the character variety comes from a flat connection. This is (a special case of) the \emph{Riemann--Hilbert correspondence} between the moduli space of flat connections and the character variety.

Therefore, our goal is to describe the character variety $\Rep(\pi_1(\S),G)$ in simple terms which allow the interpretation of its points as differential operators on the surface.

One might have the following idea: since we have a nice description of differential operators in dimension 1, if we equip $\S$ with a complex structure, then we have a manifold of dimension 1 (over $\C$). This idea was carried out by Drinfeld--Beilinson \cite{beilinson2005opers}, upon work of Drinfeld--Sokolov \cite{drinfeld1984lie}, and lead to the notion of an \textbf{oper}. The problem with this approach is that the space of all opers is not big enough. In fact, the dimension of the space of opers is only half the dimension of the character variety. The description of the character variety in this paper can be seen as a generalization of the notion of an oper.

\vspace{0.5cm}
\textbf{Summary and results.}
The main theorems of the paper provide a simple and geometric description of an open subset of the character variety $\Rep(\pi_1\S,\SL_n(\C))$, the moduli space of flat $\SL_n(\C)$-connections. The geometric interpretation involves higher complex structures defined by Vladimir Fock and the author in \cite{Fock-Thomas}. Since we are not analyzing the convergence, we only obtain a formal description.

Consider a closed surface $\S$ of genus at least 2. We fix a reference complex structure on $\S$. We consider a complex vector bundle $V$ of rank $n$ over $\S$ of degree zero (since we are interested in flat connections). Our goal is to describe the gauge classes of flat connections on $V$. We stress that we are interested in smooth connections, not necessarily holomorphic.

A connection is flat if and only if it allows a basis of flat sections. So to assure flatness, it is sufficient to find a basis of solutions to the equation $(d+A)\Psi=0$. The main idea for our description is to use the WKB method which allows to analyze solutions to differential equations with semi-classical parameter $h$. In order to get a description which is as nice as possible, we have to modify the WKB ansatz. We call this modification \emph{rational WKB method} since we will deal with fractional powers of $h$.

The description we get in this paper uses several steps:
\begin{enumerate}
	\item First we reduce a flat connection to a standard form using the gauge freedom, getting as close as possible to an oper. For $\SL_n(\C)$, the standard form is parametrized by two sets of variables $(\hat{t}_2,...,\hat{t}_n)$ and $(\hat{\mu}_2,...,\hat{\mu}_n)$. The flatness gives compatibility conditions between the $\hat{t}_k$ and $\hat{\mu}_k$.
	\item We add a formal parameter $h$ and consider a 1-parameter family of $h$-connections. This allows to use the WKB method and pushes higher order derivatives into higher orders of $h$. Our parameters become functions of $h$. The constant term of $\hat\mu_k(h)$ is denoted by $\mu_k$.
	\item The compatibility conditions are linear in $\hat{\mu}_k$ (and its derivatives), but are non-linear in the $\hat{t}_k$. To push the non-linear terms into higher orders, we impose 
	\begin{equation*}
	\hat{t}_k(h) = ht_k+\mathcal{O}(h^2) \;\forall\,  k.
	\end{equation*}
	\item Running the usual WKB method with this constraint fails. But a modified version with fractional powers of $h$ works.
\end{enumerate}

Using these steps, we can extract tensors out of flat connections (Theorem \ref{mainthm-1} below):
\begin{fakethm}\label{mainthm-intro}
For an open subset of $\Rep(\pi_1\S,\SL_n(\C))$, we can associate to a flat connection a collection of tensors $(t_k, \mu_k)$, where $t_k$ is of type $(k,0)$ and $\mu_k$ of type $(1-k,1)$, satisfying for all $k$:
\begin{equation}\label{condition-C}
(-\bar{\partial}\!+\!\mu_2\partial\!+\!k\partial\mu_2)t_{k}+\sum_{l=1}^{n-k}((l\!+\!k)\partial\mu_{l+2}+(l\!+\!1)\mu_{l+2}\partial)t_{k+l}=0 ~.
\end{equation}
\end{fakethm}

We refer to the constraints in Equation \eqref{condition-C} as \textbf{conditions $(\mathcal{C})$}. Note that for $\mu_k=0$ for all $k$, we simply get $\delbar t_k=0$. This is the subset of the character variety described by opers.

The main advantage of this description of the character variety is that \emph{it allows a geometric interpretation}. 
These conditions $(\mathcal{C})$ also appear in the description of the cotangent bundle of the moduli space of \emph{higher complex structures} $\cotang$. For $n=2$ this is nothing but the \emph{cotangent bundle of Teichmüller space}. 

\medskip
\emph{Higher complex structures} are geometric structures on surfaces generalizing the complex structure introduced by Vladimir Fock and the author in \cite{Fock-Thomas} to get a geometric approach to Hitchin's work \cite{hitchin1992lie}.
Their construction uses the punctual Hilbert scheme of the plane. Their moduli space $\T^n$ is the space of all higher complex structures modulo hamiltonian diffeomorphisms of $T^*\S$ preserving the zero-section (called \emph{higher diffeomorphisms}). $\T^n$ has similar properties to Hitchin's component of the real character variety $\Rep(\pi_1\S,\PSL_n(\R))$ (see Theorem \ref{mainresultncomplex} below), leading to the main conjecture: \emph{the moduli space of higher complex structures $\T^n$ is canonically diffeomorphic to Hitchin's component}. This would give a geometric approach to higher Teichmüller theory.

The tensors constructed in Theorem \ref{mainthm-intro} allow to link higher complex structures to the character variety (Theorem \ref{mainthm-2} below):
\begin{fakethm}\label{thm-mainmap}
There is a well-defined formal map 
\begin{equation}\label{mainmap-intro}
\omega: \cotang\to \Rep(\pi_1\S,\SL_n(\C))
\end{equation}
whose image is an open subset of the character variety. Since both manifolds have the same dimension, $\omega$ is locally a homeomorphism on the set of convergence.
\end{fakethm}
This theorem gives a sort of inverse to Theorem \ref{mainthm-intro} since it allows to reconstruct the flat connection out of the extracted tensors. Notice that in this paper, we are not analyzing the convergence of this reconstruction. 
%In the context of the WKB method, the convergence is an important issue leading to combinatorial structure like spectral networks.

Using the map $\omega$, we conjecture the following in analogy with Hitchin's section:
\begin{fakeconj}
The restriction of $\omega$ to the zero-section $\T^n\subset \cotang$ has its image inside the real character variety $\Rep(\pi_1\S,\SL_n(\R))$.
\end{fakeconj}
Admitting this conjecture, it is easy to show that the image is precisely Hitchin's component which would prove the conjecture about the equivalence between Hitchin's component and the moduli space of higher complex structures.

\vspace{0.5cm}
\textbf{Further research.}
The rational WKB method and the link between higher complex structures and flat connections open several links to related topics, which we shortly discuss here.

\vspace{0.3cm}
\textbf{\textit{Twistor description.}}
Theorems \ref{mainthm-intro} and \ref{thm-mainmap} are part of a larger, partially conjectural, picture. Hitchin's component carries a natural symplectic structure (Goldman's symplectic structure \cite{goldman1984symplectic}) and the moduli space of higher complex structures $\T^n$ carries a natural complex structure. Conjecturally both together yield a Kähler structure (admitting that both spaces are diffeomorphic). By a theorem of Feix--Kaledin \cite{feix2001hyperkahler, kaledin1997hyperkaehler}, this implies a \emph{hyperkähler structure} near the zero-section of $\cotang$. The \emph{twistor space} of this conjectural hyperkähler structure is quite similar to the twistor space of the moduli space of Higgs bundles which uses Deligne's $\lambda$-connections (see \cite[Section 4]{simpson1996hodge}). In this context, the map in \eqref{mainmap-intro} can be seen as an analog of the \emph{non-abelian Hodge correspondence}. 

%Indeed, the hyperkähler manifold can be recovered in the twistor space by the set of real twistor lines (see \cite[Theorem 3.3]{hitchin1987hyperkahler}). For the moduli space of Higgs bundles $\mathcal{M}_H$, this means that to a Higgs bundle $(V,\Phi)$, we can associate a unique real twistor line which is of the form $\lambda \Phi+A+\lambda^{-1}\Phi^*$, where $A$ is some unitary connection and $\Phi^*$ the conjugated Higgs bundle. For $\lambda=1$ we get a flat connection. This gives a map from $\mathcal{M}_H$ to the character variety which is known as the \textbf{non-abelian Hodge correspondence}.

%Our map from Theorem \ref{thm-mainmap} has the same flavor. The difference is that we are not working in a holomorphic setting (with a fixed Riemann surface), the moduli space of Higgs bundles is replaces by $\cotang$ and the image is an open subset of the character variety.

\vspace{0.3cm}
\textbf{\textit{Spectral Networks.}}
The main technical point we completely omit in this paper is the question of convergence of the rational WKB method. For the usual WKB method, the work of Gaiotto--Moore--Neitzke \cite{gaiotto2013spectral} show the emergence of a combinatorial structure, \emph{spectral networks}, which is roughly speaking a labeled graph on the surface. Given an angle $\theta\in[0,2\pi)$, the WKB series converges (using the Borel summation) in a half plane (given by $\theta$) for every point of the surface apart from the spectral network (see also \cite[Section 2]{hollands2020exact}).

\begin{oq}
Is there a combinatorial structure, generalizing spectral networks, describing the convergence for the rational WKB method? 
\end{oq}

For trivial higher complex structure (all $\mu_k=0$), we recover opers for which the usual WKB method works. This is why we expect to get a generalization of spectral networks. 
In the usual setting for spectral networks, you can choose the complex structure on $\S$ and then you pick holomorphic differentials of degree 2 to $n$. Hence, the number of degrees of freedom is less than the dimension of the character variety $\Rep(\pi_1\S,\SL_n(\C))$ for $n>2$. In our case, we should get more degrees of freedom.

%Finally, the WKB analysis is linked to the idea of \emph{abelianization}, the idea to replace a flat connection of rank $n$ by a flat connection on a line bundle over a branched cover of $\S$, called the \emph{spectral curve}. One might ask whether a notion of spectral curve emerges using the rational WKB method. For higher complex structures, there is a notion of spectral curve associated to a point of $\cotang$ (see Section \ref{cotang-spectralcurve} below).

\vspace{0.3cm}
\textbf{\textit{Topological recursion.}}
The usual WKB method applied to the differential equation defined by an oper gives a solution through a recursive method. In \cite{bouchard2017reconstructing} and \cite{eynard2021quantization}, it is shown that an explicit solution to all orders is given by \emph{topological recursion}.
This method, introduced by Eynard--Orantin (e.g. in \cite{eynard2007invariants}), produces differential forms via a recursion on the Euler characteristic out of a spectral curve and some extra data.

\begin{oq}
Is possible to get explicit solutions to the rational WKB method using topological recursion?
\end{oq}

The spectral curve of topological recursion is a branched covering of a Riemann surface. For the Hitchin spectral curve one gets a direct link to opers \cite{dumitrescu2017interplay}. In our setting, we have no fixed complex structure. So the question is to adapt topological recursion to a non-holomorphic setting. In terms of differential equations, instead of having only one differential equation for a holomorphic function, we have a system of two equations (see \eqref{flat-section-h-sln} below).

\vspace{0.3cm}
\textbf{\textit{$W$-algebras.}}
The natural transformation algebra of $\mathfrak{sl}_n$-opers is called a \emph{$W_n$-algebra}. For $n=2$ this is the famous \emph{Virasoro algebra}. It can be defined to be the function space of differential operators in dimension 1 (on the circle, or a formal punctured disk). 
%See the introduction of \cite{bouwknegt1993w} for an overview on $W$-algebras.

In our description, we have two sorts of transformations: higher diffeomorphisms and $W_n$-transformations. Starting at a point where all parameters $\mu_k$ and $t_k$ are zero, we can act by $W_n$-transformations to get non-zero differentials $t_k$, while staying within $\mu_k=0 \;\forall\, k$. Acting via higher diffeomorphisms gives the opposite. This clarifies a picture described in \cite[Section 4]{bilal1991origin}.

\vspace{0.3cm}
\begin{center}
\begin{tikzpicture}
\node[draw] (A) at (-6,0) {$(t=0, \mu \neq 0)$};
\node[draw] (B) at (0,0) {$(t=0,\mu=0)$};
\node[draw] (C) at (6,0) {$(t\neq 0, \mu=0)$};
\draw (-3,0.3) node {higher};
\draw (-3,-0.3) node {diffeomorphism};
\draw (3,0.3) node {$W_n$-transform};

\draw[->,>=latex] (B)--(A);
\draw[->,>=latex] (B)--(C);
\end{tikzpicture}
\end{center}

\vspace{0.3cm}
Another appearance of $W$-algebras comes when the surface $\S$ has boundary. Then the action of the gauge group on the space of all connections is not hamiltonian any more, due to boundary terms. These boundary terms should be a representation of the $W_n$-algebra.

\vspace{0.5cm}
\textbf{Notations.}
Throughout the paper, we denote by $\Sigma$ a smooth surface, closed, orientable and of genus at least 2. A reference complex coordinate system on $\S$ is denoted by $(z,\bar{z})$. We consider a complex bundle $V$ of rank $n$ and degree zero over $\S$.

\vspace{0.5cm}
\textbf{Acknowledgments.}
I warmly thank Vladimir Fock for many fruitful discussions and ideas. I'm also grateful to Ga\"etan Borot for helpful discussions.
I gratefully acknowledge support from the Max-Planck Institute for Mathematics in Bonn. 
\vspace{0.3cm}

\section{Flat connections in small rank}

The way we describe the space of flat $\SL_n(\C)$-connections is best illustrated by looking at the small cases for $n=2$ and $n=3$. All the difficulties arise there. The hurried reader may skip this part and directly go to the general case in Section \ref{general-case}.

\subsection{Getting started}

Consider a complex bundle $V$ of rank 2 and degree 0 over a surface $\S$. If $\S$ is a Riemann surface and $V$ a holomorphic bundle, there is the notion of an \emph{oper}.
Roughly, a $\mathfrak{sl}_2$-oper is a holomorphic connection (hence flat) locally of the form 
\begin{equation}\label{oper-sl2}
d+\begin{pmatrix} &t \\ 1& \end{pmatrix}.
\end{equation}
Empty entries in matrices are always filled with zeroes. Such a matrix is called \textit{Frobenius matrix}, or \textit{companion matrix}.

Start from a smooth $\SL_2(\C)$-connection $D$ which is flat. Locally we can write $D=d+A_1+A_2$, where $A_1$ denotes the $(1,0)$-part and $A_2$ the $(0,1)$-part. Then, we use a gauge transformation to transform $A_1$ into the Frobenius form \eqref{oper-sl2}. Such a gauge exists for a generic connection. After the gauge transformation, the connection takes the following form:
\begin{equation}\label{can-form-sl2}
d+\begin{pmatrix} &\hat{t} \\ 1& \end{pmatrix}+\begin{pmatrix} -\frac{1}{2}\del\hat{\mu}& -\frac{1}{2}\del^2\hat{\mu}+\hat{\mu}\hat{t}\\ \hat{\mu}& \frac{1}{2}\del\hat{\mu}\end{pmatrix}.
\end{equation}
Notice that there are two parameters, $\hat{t}$ and $\hat{\mu}$. To show how to get to Equation \eqref{can-form-sl2}, remark that putting $A_1$ into Frobenius form is equivalent to the existence of a basis of the space of sections $\Gamma(V)$ of the form $(v,\nabla v)$ where $v$ is a section and $\nabla=D(\partial)$ is the covariant derivative in the $z$-direction. The second matrix $A_2$ is given by the action of $\bar{\nabla}:= D(\delbar)$ on the basis $(v,\nabla v)$. The first column comes from $\bar\nabla v$, and the second column of $A_2$ can be obtained via $\bar{\nabla}\nabla v = \nabla \bar\nabla v$ using the flatness of the connection. Finally, since we work with traceless connections, we get an expression for the diagonal entries of $A_2$.

There is one more condition to ensure that the connection in \eqref{can-form-sl2} is flat. It is a compatibility condition between our two parameters:
%\textcolor{red}{check sign of $\del^3\mu$ in the sequel}
\begin{equation}\label{flatness-sl2}
(-\delbar+\hat{\mu}\del+2\del\hat\mu)\hat{t}-\frac{1}{2}\del^3\hat\mu = 0.
\end{equation}

Another way to get this condition goes as follows. We look for flat sections $\Psi$. Since $A_1$ is in Frobenius form, we have $\Psi = \binom{-\del\psi}{\psi}$. The flatness gives two equations for $\psi(z,\bar{z})$:
\begin{equation}\label{flat-section-sl2}
\left \{ \begin{array}{cl}
(\del^2-\hat{t})\psi &= 0  \\
(-\delbar-\frac{1}{2}\del\hat{\mu}+\hat{\mu}\del)\psi &= 0.
\end{array} \right.
\end{equation}
Recall that a system of two differential equations $D_1\psi=0$ and $D_2\psi=0$ where $D_1$ and $D_2$ are linear differential operators is \emph{compatible} if $$[D_1,D_2] = 0 \mod \langle D_1, D_2 \rangle$$ where $\langle D_1, D_2\rangle$ is the left ideal generated by $D_1$ and $D_2$ in the space of differential operators. This is a way to formulate the Frobenius integrability condition. If that condition is not satisfied, it is possible to reduce the system to a smaller one. This compatibility condition is equivalent to the associated connection being flat, i.e. we get Equation \eqref{flatness-sl2}.

With what we have done, we could describe $\Rep(\pi_1\S,\SL_2(\C))$ as the space of $(\hat{\mu},\hat{t})$ satisfying Equation \eqref{flatness-sl2}. This is unsatisfying for several reasons:
\begin{itemize}
	\item The parameters $\hat\mu$ and $\hat{t}$ are not tensors. Their global nature on the surface is complicated.
	\item The condition from Equation \eqref{flatness-sl2} contains a higher order derivative $\frac{1}{2}\del^3\hat\mu$.
\end{itemize}
Indeed, considering a holomorphic coordinate change $z\mapsto w(z)$, the parameter $\hat{t}(z)$ transforms as
\begin{equation}\label{cp1-transfo}
\hat{t}(w) = \hat{t}(z)\left(\frac{dz}{dw}\right)^2+\frac{1}{2}S(w,z)
\end{equation}
where $S(w,z) = \frac{w'''}{w'}-\frac{3}{2}\left(\frac{w''}{w'}\right)^2$ denotes the Schwarzian derivative, where we have put $w'=\frac{dw}{dz}$. This transformation rule shows that $\hat{t}$ defines a \emph{complex projective structure}\footnote{A complex projective structure on a surface is an atlas where the charts are open sets of $\mathbb{C}P^1$ and the transition functions are Möbius transformations. Using the uniformization theorem, the space of complex projective structures can be identified with the cotangent bundle of Teichmüller space.} on $\S$ (see for example \cite[Section 2]{bilal1991origin}). 
%In conformal field theory language, one can say that $\hat{t}$ defines the energy-momentum tensor.

There is a way to solve both of these problems at once: we introduce a formal parameter $h$ and consider a family of $h$-connections. This means that we consider
\begin{equation}\label{can-form-h-sl2}
h d+\begin{pmatrix} &\hat{t}(h) \\ 1& \end{pmatrix}+\begin{pmatrix} -\frac{1}{2}h\del\hat{\mu}& -\frac{1}{2}h^2\del^2\hat{\mu}+\hat{\mu}\hat{t}\\ \hat{\mu}(h)& \frac{1}{2}h\del\hat{\mu}\end{pmatrix}.
\end{equation}
Note that all the derivatives come with an $h$ and that our parameters $\hat{t}(h)$ and $\hat\mu(h)$ now depend on the parameter $h$. We think of $h$ as a deformation parameter which is ``small''. Taylor developing our parameters we get
\begin{equation}\label{t-mu}
\hat{t}(h) = t+\mathcal{O}(h) \;\; \text{ and }\;\; \hat\mu(h) = \mu+\mathcal{O}(h).
\end{equation}

The good news is that $t$ and $\mu$, the constant terms of $\hat{t}$ and $\hat\mu$, are tensors:
\begin{prop}
The parameters $t$ and $\mu$ are tensors of type $(2,0)$ and $(-1,1)$ respectively.
\end{prop}
\begin{proof}
For $t$, take the transformation rule for $\hat{t}$ from Equation \eqref{cp1-transfo}, together with the formal parameter $h$, and consider the $h^0$-level. This gives $t(w) = t(z)\left(\frac{dz}{dw}\right)^2$ since the Schwarzian derivative contributes with a factor $h^2$.

For $\mu$, from Equation \eqref{flatness-sl2}, we see that $\delbar t$ and $\del^3\mu$ are of the same kind. Since $t$ is of type $(2,0)$, it follows that $\mu$ is of type $(-1,1)$.
\end{proof}

This solves the first of the problems listed above. The second problem, the higher order derivative term in the constraint \eqref{flatness-sl2}, gets solved by using a method for solving differential equations with formal parameter, the \emph{WKB method}.

\subsection{WKB method}

Named after G. Wentzel, H. Kramers and L. Brillouin, the WKB method was developed in 1926 for approximating solutions to differential equations. 
In the 1980s it was developed into an exact method (see for example \cite{voros1983return}). More recently it is used in the context of flat connections, for example in \cite{gaiotto2013spectral} or \cite{iwaki2014exact}.

It is a method to solve differential equations depending on $h$-derivatives on an unknown $\psi$ in the semiclassical limit (when $h \to 0$). The WKB method uses the following ansatz: 
\begin{equation}\label{WKB-ansatz}
\psi = \exp\left(\frac{1}{h}(s_0+h s_1+h s_2+...)\right)
\end{equation}
The best way to see how the method works is to see an example.

\begin{example}
Consider the following differential equation, known as the \emph{Schrödinger equation}:
\begin{equation}\label{Schroedinger}
(h^2\del^2-\hat{t}(h))\psi=0.
\end{equation}
Putting the ansatz $\psi=\exp\left(\frac{1}{h}s\right)$ with $s=s_0+h s_1+...$, we get 
\begin{equation}\label{Schroedinger-WKB}
(\del s)^2+h \del^2s - \hat{t}(h) = 0.
\end{equation}
Write $\hat{t}=t +\mathcal{O}(h)$, then the $h^0$-term gives 
$$(\del s_0)^2-t = 0.$$
So we get two solutions $\del s_0=\pm \sqrt{t}$. The important feature of the WKB ansatz is that knowing $\hat{t}(h)=\sum_i h^i t^{(i)}$ to all orders, we get all $s_i$ (for $i>0$) recursively from $s_0$. For example the $h$-term of \eqref{Schroedinger-WKB} gives $2\del s_0\del s_1+\del^2 s_0-t^{(1)}=0$. In general the equation for $s_i$ is $$2\del s_0\del s_i +\textstyle\sum_{k=1}^{i-1}\del s_k\del s_{i-k} +\del^2s_{i-1}-t^{(i)}=0.$$
In total, we get two fundamental solutions, one for each choice of $\del s_0$ which is in compliance with the fact that the Schrödinger equation is of second order. The integration constants for the $s_i$ can be absorbed by an overall factor (a formal power series in $h$) which scales the fundamental solutions.
\end{example}

An important question concerning the WKB ansatz is its convergence. In general, the formal sum in Equation \eqref{WKB-ansatz} never converges. Though, it is possible through a resummation method, the \emph{Borel summation}, to make sense of Equation \eqref{WKB-ansatz} as an asymptotic series \cite{voros1983return}. 

It turns out that the convergence at a point $P$ can only be ensured when the limiting path to $P$ stays in some half-plane, defined by an angle $\theta\in [0,2\pi)$. On a Riemann surface, this gives rise to a graph depending on $\theta$, called a \emph{spectral network}. We refer to the work of \cite{gaiotto2013spectral} and \cite{hollands2020exact} for details.

In this article, we completely put aside the question of convergence and work formally.

\subsection{WKB method for flat connections}\label{description-for-n2}

We wish to use the WKB method for our problem of parametrizing the character variety $\Rep(\pi_1\S,\SL_2(\C))$. We use the fact that \textit{a connection is flat iff it admits a basis of flat sections}. 

In Equation \eqref{flat-section-sl2}, we have seen the system of differential equations we have to solve to get a flat section. In the setting of $h$-connections, the system becomes
\begin{equation}\label{flat-section-h-sl2}
\left \{ \begin{array}{cl}
(h^2\del^2-\hat{t})\psi &= 0  \\
(-h\delbar-\frac{1}{2}h\del\hat{\mu}+\hat{\mu}h\del)\psi &= 0.
\end{array} \right.
\end{equation}
Note that the first equation is the Schrödinger equation from Equation \eqref{Schroedinger}.

Using the WKB ansatz $\psi=\exp\left(\frac{1}{h}s\right)$ with $s=s_0+h s_1+...$, we get
\begin{equation}\label{flat-section-sl2-WKB}
\left \{ \begin{array}{cl}
(\del s)^2+h \del^2 s-\hat{t} &= 0  \\
-\delbar s-\frac{1}{2}h\del\hat{\mu}+\hat{\mu}\del s &= 0.
\end{array} \right.
\end{equation}

The big difference to the example of the Schrödinger equation is that we have a system of two differential equations which gives compatibility conditions. The $h^0$-term gives
\begin{equation}\label{flat-section-sl2-WKB-s0}
\left \{ \begin{array}{cl}
(\del s_0)^2-t &= 0  \\
-\delbar s_0+\mu\del s_0 &= 0.
\end{array} \right.
\end{equation}
The compatibility condition for this system reads
\begin{equation}\label{cond-C-n2}
(-\delbar+\mu\del+2\del\mu)t=0.
\end{equation}

Notice the similarity, but also the differences between this equation and the constraint \eqref{flatness-sl2} from above. The quantities involved now are tensors, and we got ride of the higher order derivative term since it appears in a higher order of $h$.

The WKB strategy now tells us that under the condition \eqref{cond-C-n2}, we can solve for $s_0$, which gives us two solutions. Then we should be able to get all $s_i$ recursively. The problem is that we will get a compatibility condition for each $h^k$-term of the system \eqref{flat-section-sl2-WKB}. 

The way around this problem is to consider the higher order compatibility conditions as equations for the higher order terms of $\hat\mu$. To be more precise:
\begin{prop}
Knowing all higher order terms of $\hat{t}$ (with $t\neq 0$) and the constant term $\mu$ of $\hat\mu$ such that the compatibility condition \eqref{cond-C-n2} is satisfied, we can recursively solve for all $s_i$ in the WKB ansatz.
\end{prop} 
\begin{proof}
Put $\hat{t}(h)=\textstyle\sum_i t^{(i)}h^i$ and $\hat\mu(h) = \textstyle\sum_i \mu^{(i)}h^i$. The system \eqref{flat-section-sl2-WKB} for $s_i$ ($i>0$) reads
\begin{equation}
\left \{ \begin{array}{cl}
2\del s_i \del s_0+\textstyle\sum_{k=1}^{i-1}\del s_k \del s_{i-k}+\del^2 s_{i-1}-t^{(i)} &= 0  \\
-\delbar s_i-\frac{1}{2}\del\mu^{(i-1)}+\textstyle\sum_{k=0}^i \mu^{(k)}\del s_{i-k} &= 0.
\end{array} \right.
\end{equation}
If $t^{(i)}$ is known for all $i$, then we have one new variable $\mu^{(i)}$ which appears in the compatibility condition. The factor of $\mu^{(i)}$ is $\del s_0$ which has to be non-zero. This means that the constant term of $\hat{t}$ has to be non-zero.
\end{proof}

Another good news is that we can actually give $\hat{t}$ to all orders: the $\hat{t}$ as a global object on $\S$ defines a complex projective structure. The tensor $\mu$ is of type $(-1,1)$, hence can be interpreted as Beltrami differential (for a generic flat connection we will have $\mu\bar{\mu}\neq 1$). Hence we get a complex structure on $\S$. Then, there is a preferred complex projective structure, coming from the Poincaré uniformization theorem. We take $\hat{t}$ to be this structure, depending only on a holomorphic quadratic differential $t$. 

The last important point is that the compatibility condition $(-\delbar+\mu\del+2\del\mu)t=0$ is precisely the statement that \emph{$t$ is holomorphic with respect to the complex structure defined by the Beltrami differential $\mu$}.

Finally, notice that we have not used the full gauge freedom when we decided to put $A_1$ into Frobenius form. This means that we can associate different pairs $(t,\mu)$ to the same flat connection. It turns out that all these pairs are equivalent under diffeomorphisms of $\S$ isotopic to the identity (for the general case, see step 3 in the proof of Theorem \ref{mainthm-2}).

Therefore, we get a parametrization of an open subset of $\Rep(\pi_1\S,\SL_2(\C))$ which has a geometric meaning: a couple $(\mu,t)$ of a Beltrami differential and a holomorphic quadratic differential modulo isotopies. \emph{This is precisely the cotangent bundle of Teichmüller space!}

More precisely, the work of Gallo--Kapovich--Marden \cite{gallo2000monodromy} shows that an open dense subset of the $\SL_2(\C)$-character variety is described as the monodromy of a complex projective structure. 
%Teichmüller space has a Kähler structure, so its cotangent bundle, denoted by $T^*\T^2$, has a hyperkähler structure near the zero-section. Its associated twistor space links $T^*\T^2$ to that open subset of $\Rep(\pi_1\S,\SL_2(\C))$. 

\subsection{Case for \texorpdfstring{$n=3$}{n=3}}

We use the same strategy as for the case $n=2$. We start with a family of flat $h$-connections of the form $hd+A_1(h)+A_2(h)$. We use the gauge freedom to put $A_1$ into Frobenius form. This is possible for a generic connection. We then get a connection of the form
\begin{equation}\label{can-form-sl3}
hd+\begin{pmatrix} && \hat{t}_3(h)\\ 1&&\hat{t}_2(h)\\ &1& \end{pmatrix}+\begin{pmatrix} \hat\mu_1(h) &*&* \\ \hat\mu_2(h)&*&*\\ \hat\mu_3(h)& *&* \end{pmatrix}
\end{equation}
where the stars are explicit expressions in the parameters $(\hat{t}_2,\hat{t}_3,\hat\mu_2,\hat\mu_3)$. Also $\hat\mu_1$ can be expressed by the parameters since the second matrix has trace zero: 
\begin{equation}\label{mu1-n3}
\hat\mu_1=-\frac{2}{3}\hat\mu_3 \hat{t}_2-h\del\hat\mu_2-\frac{2}{3}h^2\del^2\hat\mu_3.
\end{equation}

The flatness of the connection gives two compatibility conditions among these parameters. These conditions contain higher order derivatives and quadratic terms in $\hat{t}_2$. 

Denote by $t_2,t_3, \mu_2$ and $\mu_3$ the constant terms of the parameters. As before, these quantities are tensors: $t_k$ is of type $(k,0)$ and $\mu_k$ is of type $(1-k,1)$ for $k\in\{2,3\}$ (see Proposition \ref{prop-nature-t-mu} below).

To describe flat sections, recall that we act on the right: $\Psi^T(d+A_1+A_2)=0$. We get $\Psi^T=(\psi \; \scalebox{0.75}[1.0]{\( - \)}\del\psi \; \del^2\psi)$. The system of differential equations for $\psi$ reads
\begin{equation}\label{flat-section-h-sl3}
\left \{ \begin{array}{cl}
(h^3\del^3-\hat{t}_2h\del+\hat{t}_3)\psi &= 0  \\
(-h\delbar-\hat{\mu}_1+\hat{\mu}_2h\del-\hat{\mu}_3h^2\del^2)\psi &= 0.
\end{array} \right.
\end{equation}

Using the WKB ansatz we get
\begin{equation}\label{flat-section-h-sl3-WKB}
\left \{ \begin{array}{cl}
(\del s)^3+3h (\del s)\del^2s+h^2\del^3s-\hat{t}_2\del s+\hat{t}_3 &= 0  \\
-\delbar s-\hat{\mu}_1+\hat{\mu}_2\del s-\hat\mu_3((\del s)^2+h\del^2s) &= 0.
\end{array} \right.
\end{equation}

The $h^0$-term, using Equation \eqref{mu1-n3} for $\hat\mu_1$, gives the system 
\begin{equation}
\left \{ \begin{array}{cl}
(\del s_0)^3-t_2\del s_0+t_3 &= 0  \\
-\delbar s_0+\frac{2}{3}\mu_3t_2+\mu_2\del s_0-\mu_3(\del s_0)^2 &= 0.
\end{array} \right.
\end{equation}
The compatibility condition of the system is
\begin{equation}\label{n3-wrong}
(\delbar-\mu_2\del-3\del\mu_2)t_3 - \frac{2}{3}t_2\del(\mu_3t_2) + \del s_0((-\delbar+\mu_2\del+2\del\mu_2)t_2+2\mu_3\del t_3+3\del\mu_3 t_3) =0.
\end{equation}

There are two problems here: 
\begin{itemize}
	\item We wish to have two separate conditions, one for $t_2$ and one for $t_3$, in analogy with the two constraint from the flatness. There is no reason why the two terms in Equation \eqref{n3-wrong} are separately zero.
	\item We wish to get rid of the quadratic term $\frac{2}{3}t_2\del(\mu_3t_2)$.
\end{itemize}

To solve the second problem, the solution is to declare that $\hat{t}_k$ has no constant term and starts only at order 1:
\begin{equation}\label{condition-on-t}
\hat{t}_k=h t_k +\mathcal{O}(h^2) \text{ for } k\in \{2,3\}.
\end{equation}

This shift pushes higher order terms in the $\hat{t}_k$ into higher orders in $h$. So it solves the second problem, but an even bigger problem arises with this constraint on $\hat{t}_k$: the usual WKB method does not work any more.

This can already be seen for $n=2$. If we assume $\hat{t}=h t+\mathcal{O}(h^2)$, the system \eqref{flat-section-h-sl2} gives first that $s_0$ is a constant, and then we get $t=0$. This is impossible since $t$ should be a free parameter.

The solution to this problem is to use a modified version of WKB, which we call \textbf{rational WKB}. We keep the same ansatz $\psi=\exp\left(\frac{1}{h}s\right)$ but now the function $s$ depends on a \emph{fractional} exponent of $h$. More precisely, for $\SL_n(\C)$, we have
$$s=s_0+h^{1/n}s_{1/n}+h^{2/n}s_{2/n}+...$$
We will see that this solves both problems: for $n=3$, we will get two compatibility equations, without the quadratic term in $t_2$. Even for $n=2$, this new viewpoint slightly changes the computation. 

\begin{example}
The equations describing a flat section using the modified WKB ansatz $\psi=\exp\left(\frac{1}{h}s\right)$ are still given by Equation \eqref{flat-section-sl2-WKB}.
Now, the function $s$ depends on $h^{1/2}$ and $\hat{t}=ht+\mathcal{O}(h^2)$. The $h^0$-term gives $\del s_0=0$ and $\delbar s_0+\mu \del s_0=0$. Hence $s_0$ is a constant. Then, we get two equations for $s_{1/2}$:
\begin{equation*}
\left \{ \begin{array}{cl}
(\del s_{1/2})^2-t &= 0  \\
-\delbar s_{1/2}+\mu\del s_{1/2} &= 0.
\end{array} \right.
\end{equation*}
The compatibility condition of this system is precisely Equation \eqref{cond-C-n2}.
All other compatibility conditions can be seen as defining equations for the higher order terms of $\hat\mu$.
\end{example}

\begin{Remark}
For $n=2$ and $\mu=0$, \cite[Lemma 3]{fock2008cosh} gives all higher order terms for $\hat{t}$ and $\hat\mu$. The result is $$\hat{t}(h)=ht+h^2((\del\varphi)^2-\del^2\varphi) \;\text{ and }\; \hat\mu(h)=-h\bar{t}e^{-2\varphi}$$ where $e^\varphi$ is the metric and satisfies the cosh-Gordon equation $$\frac{1}{2}\del\delbar\varphi=e^\varphi+t\bar{t}e^{-\varphi}.$$
\end{Remark}

Let us return to $n=3$ where we have the system \eqref{flat-section-h-sl3-WKB}. From $\hat{t}_k=ht_k+\mathcal{O}(h^2)$, we get that $s_0$ is constant. The equations for $s_{1/3}$ are given by
\begin{equation*}
\left \{ \begin{array}{cl}
(\del s_{1/3})^3+t_3 &= 0  \\
-\delbar s_{1/3}+\mu_2\del s_{1/3} &= 0.
\end{array} \right.
\end{equation*} 
The compatibility equation for this system reads 
\begin{equation}\label{cond-C-n3-1}
(-\delbar+\mu_2\del+3\del\mu_2)t_3=0.
\end{equation}

The equations for $s_{2/3}$ are given by 
\begin{equation*}
\left \{ \begin{array}{cl}
3(\del s_{1/3})^2\del s_{2/3}-t_2 \del s_{1/3} &= 0  \\
-\delbar s_{2/3}+\mu_2\del s_{2/3}-\mu_3 (\del s_{1/3})^2 &= 0
\end{array} \right.
\end{equation*} 
leading to the compatibility condition
\begin{equation}\label{cond-C-n3-2}
(-\delbar+\mu_2\del+2\del\mu_2)t_2+3\del\mu_3t_3+2\mu_3\del t_3=0.
\end{equation}

The two constraints \eqref{cond-C-n3-1} and \eqref{cond-C-n3-2} have no higher derivative terms and are linear in the $\mu$'s and $t$'s (and their derivatives). Further, we get exactly the constraints from Equation \eqref{n3-wrong}, but separately and without the quadratic term.

Finally, the most important advantage of this description of $\Rep(\pi_1\S,\SL_3(\C))$ is that it allows a \emph{geometric interpretation}. Conditions \eqref{cond-C-n3-1} and \eqref{cond-C-n3-2} are precisely the compatibility conditions appearing in the \emph{cotangent bundle of the moduli space of higher complex structures of order 3}. This generalizes our observation for $n=2$ where we found the cotangent bundle of Teichmüller space. This geometric interpretation is discussed in detail in Section \ref{geometric-description}.

\section{Rational WKB and flat \texorpdfstring{$\SL_n(\C)$}{SL(n,C)}-connections}\label{general-case}

In this section we generalize the construction of the previous section to arbitrary $n$. In particular we will associate tensors to flat connections of an open subset of the character variety $\Rep(\pi_1\S,\SL_n(\C))$. The geometric meaning of this description is discussed in Section \ref{geometric-description}.

We start with a smooth flat $\SL_n(\C)$-connection $D$ in a complex bundle $V$ of degree zero over $\S$. We fix a reference complex structure on $\S$ which allows to use complex coordinates. Further, we consider a family of $h$-connections $D_h$ such that $D_1=D$, where $h$ denotes a formal parameter. We write $D_h$ in a local chart as $hd+A_1+A_2$, where $A_1$ denotes the $(1,0)$-part and $A_2$ the $(0,1)$-part of the connection form. Further, we put $\nabla = D_h(\partial)$ and $\bar\nabla= D_h(\bar{\partial})$, the covariant derivatives in the $z$ and $\bar{z}$-direction. 

Using a gauge transform, we can assume $A_1$ in Frobenius form (a companion matrix). The local existence of such a gauge for a generic connection can be found in \cite[Proposition B.1]{thomas2020higher}. 
In global terms, we assume the existence of a filtration $\mathcal{F}$ of $V$ such that $$\nabla:\mc{F}_i/\mc{F}_{i-1} \rightarrow \mc{F}_{i+1}/\mc{F}_i$$
is an isomorphism for all $i$. This is equivalent to the existence of a global section $v$ such that $B=(v,\nabla v,\nabla^2v,...,\nabla^{n-1}v)$ forms a basis of the space of sections $\Gamma(V)$.
The subset of connections satisfying this condition is clearly open (since $\S$ is compact) and non-emtpy (since opers are inside).

Locally the connection $D_h$ can be written as follows:
\begin{equation}\label{can-form-general}
D_h = hd+\begin{pmatrix} &&& \hat{t}_n(h)\\ 1&&&\vdots\\ & \ddots &&\hat{t}_2(h) \\ &&1&0 \end{pmatrix}+\begin{pmatrix} \hat\mu_1(h) &*&* & * \\ \hat\mu_2(h) &*&* & * \\ \vdots &*&*&*\\ \hat\mu_n(h)& *&* &* \end{pmatrix}.
\end{equation}
Empty entries denote zeroes and the stars are entries which are explicit expressions in the other parameters. 
%We denote the entries of the second matrix by $\hat\alpha_{i,j}$ such that $\hat\alpha_{k,1}=\hat\mu_k$. 
The fact that the second matrix is trace-free gives an explicit formula for $\hat\mu_1$. Hence, the connection $D_h$ is described by $(\hat{t}_k(h), \hat\mu_k(h))_{2\leq k \leq n}$.

Let us give an explanation why all entries are determined from the $\hat{t}_k$ and $\hat\mu_k$. Putting the first matrix into Frobenius form is equivalent to consider the basis $B$ above. Then we can express $\nabla^nv$ and $\bar\nabla v$ in this basis:
\begin{equation}\label{eqqq1}
\left \{ \begin{array}{cl}
\nabla^n v &= \hat{t}_{n}v + \hat{t}_{n-1} \nabla v+...+\hat{t}_{2}\nabla^{n-2}v \\
\bar{\nabla}v &= \hat{\mu}_1 v + \hat{\mu}_2 \nabla v + ... + \hat{\mu}_n \nabla^{n-1}v.
\end{array} \right.
\end{equation}
Since $D_h$ is flat, we have $\bar\nabla\nabla^k v = \nabla^k\bar\nabla v$, from which we can compute $A_2$ using Equation \eqref{eqqq1}.

The flatness of $D_h$ gives $n-1$ compatibility conditions between the parameters $(\hat{t}_k(h), \hat\mu_k(h))$. Indeed it is clear that the curvature of a connection of the form \eqref{can-form-general} has rank 1. These constraints are complicated expressions containing higher order derivatives and higher order terms in the $\hat{t}_k$. Our procedure will simplify these constraints.

To assure that $D_h$ is flat, it is sufficient to assure the existence of a basis of flat sections. Recall that we act on the right: $\Psi^T(d+A_1+A_2)=0$. Since the first matrix is in Frobenius form, the components of $\Psi$ are given by $\Psi_k=(-1)^k\del^k\psi$. Hence, a flat section is uniquely described by a local function $\psi(z,\bar{z})$, which globally is section of $K^{-(n-1)/2}$. The flatness is then reduced to the following system of differential equations:
\begin{equation}\label{flat-section-h-sln}
\left \{ \begin{array}{cl}
(h^n\del^n-\hat{t}_2h^{n-2}\del^{n-2}+\hat{t}_3h^{n-3}\del^{n-3}-...+(-1)^{n-1}\hat{t}_n)\psi &= 0  \\
(-h\delbar-\hat\mu_1+\hat\mu_2h\del-\hat\mu_3h^2\del^2+...+(-1)^{n}\hat\mu_n\del^{n-1})\psi &= 0.
\end{array} \right.
\end{equation}
The compatibility conditions of this system give exactly the $n-1$ compatibility conditions between the $\hat{t}_k$ and the $\hat\mu_k$.

To avoid higher order terms in the $\hat{t}_k$ in the compatibility conditions, we consider a family of $h$-connections such that $\frac{d\hat{t}_k}{dh}\mid_{h=0}=0$ for all $k$. Then the Taylor development of the parameters can be written as follows:
\begin{equation}\label{cond-on-parameters}
\hat{t}_k(h) = ht_k+\mathcal{O}(h^2) \;\; \text{ and }\;\; \hat\mu_k(h) = \mu_k+\mathcal{O}(h).
\end{equation}

The parameters $t_k$ and $\mu_k$ play an important role in the sequel. Their first property is that they transform like tensors:
\begin{prop}\label{prop-nature-t-mu}
The parameters $t_k$ and $\mu_k$ are tensors of type $(k,0)$ and $(1-k,1)$ respectively.
\end{prop}
\begin{proof}
To determine the global nature of $t_k$, consider the transformation of the differential operator 
$$h^n\del_z^n-\hat{t}_2h^{n-2}\del_z^{n-2}+\hat{t}_3h^{n-3}\del_z^{n-3}-...+(-1)^{n-1}\hat{t}_n$$
under a holomorphic coordinate change $z\mapsto w(z)$.
We have $\del_z \mapsto \frac{dw}{dz}\del_w$. Since we consider the lowest order terms, those in which the $t_k$ appear, we see that the only contribution from $\hat{t}_{n-k}h^k\left(\frac{dw}{dz}\del_w\right)^k$ is $t_{n-k}h^{k+1}\left(\frac{dw}{dz}\right)^k\del_w^k$. In the same vain, the highest derivative term contributes $h^n\left(\frac{dw}{dz}\right)^n\del_w^n$. To get the standard form (dominant coefficient 1), we have to divide by $\left(\frac{dw}{dz}\right)^n$. Hence we get $$t_k(w) = t_k(z)\left(\frac{dz}{dw}\right)^k.$$
The same analysis works for $\mu_k$ by considering the transformation of the differential operator $$-h\delbar-\hat\mu_1+\hat\mu_2h\del-\hat\mu_3h^2\del^2+...+(-1)^{n}\hat\mu_n\del^{n-1}.$$
\end{proof}

Our strategy to solve the system \eqref{flat-section-h-sln} is to use the modified version of the WKB method, what we call \textit{rational WKB}. This allows to solve the system for low orders in $h$, and then to complete recursively to a solution. The ansatz is
\begin{equation}\label{rational-WKB-ansatz}
\psi=\exp\left(\frac{1}{h}s(z,\bar{z},h^{1/n})\right) \;\; \text{ with }\;\; s=s_0+h^{1/n}s_{1/n}+h^{2/n}s_{2/n}+...
\end{equation}
Using this ansatz for our system \eqref{flat-section-h-sln}, we get the following description for the character variety:

\begin{thm}\label{mainthm-1}
For an open subset of $\Rep(\pi_1\S,\SL_n(\C))$, we can associate to a flat connection a collection of tensors $(t_k, \mu_k)$, where $t_k$ is of type $(k,0)$ and $\mu_k$ of type $(1-k,1)$, satisfying for all $k$:
\begin{equation}\label{condition-C-2}
(-\bar{\partial}\!+\!\mu_2\partial\!+\!k\partial\mu_2)t_{k}+\sum_{l=1}^{n-k}((l\!+\!k)\partial\mu_{l+2}+(l\!+\!1)\mu_{l+2}\partial)t_{k+l}=0 ~.
\end{equation}
\end{thm}
These constraints are called ``conditions $(\mathcal{C})$''. They have a geometric meaning in the theory of higher complex structures (see Theorems \ref{conditionC} and \ref{mainthm-2} below).
%\begin{Remark}
%Using the notion of parabolic connections introduced in \cite{thomas2020higher}, we can say that to a flat parabolic connection, we can associate a collection of tensors satisfying conditions $(\mathcal{C})$.
%\end{Remark}

The idea of the proof is the following: the compatibility conditions of the lowest equations (gradation by exponents of $h$) of system \eqref{flat-section-h-sln}, using the rational WKB ansatz, give a Hamilton--Jacobi equation. This equation describes a generating function for a Lagrangian submanifold in $T^{\C}\S$. The main technical tool is to realize conditions $(\mathcal{C})$ as the condition of being Lagrangian (see Proposition \ref{spectralcurveprop} below).

\begin{proof}
Given a flat connection in the form \eqref{can-form-general}, we associate the parameters $t_k$ and $\mu_k$ from Equation \eqref{cond-on-parameters}. These are tensors by Proposition \ref{prop-nature-t-mu}. To get the conditions $(\mathcal{C})$, consider the system of differential equations \eqref{flat-section-h-sln} with the rational WKB method. We claim that the compatibility conditions for $s_{k/n}$ where $1\leq k <n$ are precisely conditions $(\mathcal{C})$. Those for $k\geq n$ involve higher order terms of $\hat{t}_k$ or $\hat{\mu}_k$ and hence can be considered as an equation on these higher order terms.

The first appearance of $s_{k/n}$ in the first equation is for the $h^{(k+n-1)/n}$-term, and in the second equation for the $h^{k/n}$-term.
Since we consider $s_{k/n}$ for $k<n$, we can neglect all terms $h^q$ with $q\geq 2$ for the first equation and $q\geq 1$ for the second equation. In particular, we can neglect the higher order terms in $$h^k\del^k\psi = (\del s)^k \psi + \text{ higher order terms}.$$ Hence the system \eqref{flat-section-h-sln} with rational WKB ansatz simplifies to
\begin{equation}\label{flat-section-simplified}
\left \{ \begin{array}{cl}
(\del s)^n-ht_2(\del s)^{n-2}+ht_3(\del s)^{n-3}-...+(-1)^{n-1}ht_n &= 0  \\
-\delbar s+\mu_2 \del s-\mu_3(\del s)^2+...+(-1)^{n}\mu_n(\del s)^{n-1} &= 0.
\end{array} \right.
\end{equation}
Now we should restrict to the system of $s_{k/n}$ and compute the compatibility condition. It turns out that we can get all the compatibility conditions at once, using the grading by $h^{1/n}$.

The crucial point is that the system \eqref{flat-section-simplified} is a Hamilton--Jacobi equation. Put $P(x) = x^n-ht_2x^{n-2}\pm...+(-1)^{n-1}ht_n$ and $Q(x,y) = -y+\mu_2 x-\mu_3x^2\pm...+(-1)^n\mu_n x^{n-1}$. The zero-set of these two polynomials is a surface $L\subset\C^2$ (we treat $x$ and $y$ as complex coordinates). We obtain system \eqref{flat-section-simplified} by $x=\del s$ and $y=\delbar s$. This means that $s$ is a generating function for $L$ which has to be Lagrangian modulo $h^2$. In other words: the compatibility condition of system \eqref{flat-section-simplified} is equivalent to $L$ being Lagrangian modulo $h^2$.

Now, a surface defined by two polynomials $P$ and $Q$ is Lagrangian iff the Poisson bracket $\{P,Q\}$ vanishes modulo the ideal generated by $P$ and $Q$. Using Proposition \ref{spectralcurveprop} below, we see that this Poisson bracket vanishes iff conditions $(\mathcal{C})$ are satisfied for $((-1)^{k-1}t_k,(-1)^k\mu_k)_{2\leq k\leq n}$. Finally, we check explicitly in conditions $(\mathcal{C})$ that apart from an overall sign, these are conditions $(\mc{C})$ for $(t_k,\mu_k)_{2\leq k\leq n}$. 
\end{proof}

To describe the geometric meaning of Theorem \ref{mainthm-1}, we have to consider higher complex structures.

\section{Higher complex structures}\label{higher-complex}

Higher complex structures are geometric structures on surfaces generalizing the complex structure. They are motivated by the idea to describe components of real character varieties, such as Hitchin components, as moduli space of some geometric structure. They were introduced in \cite{Fock-Thomas} and developed in \cite{thomas2020higher, thomas2020generalized}.

This section gives a summary of \cite{Fock-Thomas}, in particular the construction of higher complex structures, their moduli space and the cotangent bundle to its moduli space. 

\subsection{Introduction and definitions}

Before defining higher complex structures, let us describe complex structures on surfaces. A \emph{complex structure} is an atlas with charts in $\C$ and holomorphic transition functions. For a surface, this is equivalent to the existence of an endomorphism $J(z)$ of $T^*_z\S$ with $J(z)^2=-\id$ varying smoothly with $z\in\S$. The operator $J$ allows to consider $T^*_z\S$ as a complex vector space.

We wish to diagonalize $J$. To do so, we have to complexify the cotangent bundle. Since $J$ is a real operator, the eigendirections are mutually complex conjugated. Hence given a direction in $T^{*\C}_z\S$ at each point $z$, we can reconstruct $J$. 
In summary: a complex structure is \emph{entirely encoded in a section of $\bb{P}(T^{*\C}\S)$}.

Using coordinates, this description gives a tensor, the \textbf{Beltrami differential}. Given a reference complex structure on $\S$, it induces coordinates $(\del,\delbar)$ on $T^{\C}\S$. The dual of the eigendirection corresponding to the eigenvalue $i$ of $J$ is then described by $\delbar-\mu\del$ where $\mu$ is the Beltrami differential, a tensor of type $(-1,1)$.

The idea of higher complex structures is to replace the linear direction by a polynomial direction, or more precisely a $n$-jet of a curve inside $T^{*\C}\S$.
To get a precise definition, we use the \textbf{punctual Hilbert scheme} of the plane, denoted by $\Hilb^n(\C^2)$ which is defined by
$$\Hilb^n(\C^2)=\{I \text{ ideal of } \C[x,y] \mid \dim \C[x,y]/I = n\}.$$
A generic point in $\Hilb^n(\C^2)$ is an ideal whose algebraic variety is a collection of $n$ distinct points in $\C^2$. A generic ideal can be written as 
\begin{equation}\label{ideal-Hilb}
I=\langle-x^n+t_1x^{n-1}+...+t_n, -y+\mu_1+\mu_2x+...+\mu_nx^{n-1} \rangle.
\end{equation}
Moving around in $\Hilb^n(\C^2)$ corresponds to a movement of $n$ particles in $\C^2$. But whenever $k$ particles collide the Hilbert scheme retains an extra information: the $(k-1)$-jet of the curve along which the points entered into collision. The \textbf{zero-fiber}, denoted by $\Hilb^n_0(\C^2)$, consists of those ideals whose support is the origin. A generic point in $\Hilb^n_0(\C^2)$ is of the form 
$$\langle x^n, -y+\mu_2x+\mu_3x^2+...+\mu_nx^{n-1}\rangle$$
which can be interpreted as a $(n-1)$-jet of a curve at the origin.

We can now give the definition of the higher complex structure:
\begin{definition}[Def.2 in \cite{Fock-Thomas}]
 A \textbf{higher complex structure} of order $n$ on a surface $\Sigma$, in short \textbf{$n$-complex structure}, is a section $I$ of $\Hilb^n_0(T^{*\mathbb{C}}\Sigma)$ such that at each point $z\in \S$ we have $I(z)+\bar{I}(z)=\langle p, \bar{p} \rangle$, the maximal ideal supported at the origin of $T_z^{*\mathbb{C}}\Sigma$.
\end{definition}
Notice that we apply the punctual Hilbert scheme pointwise, giving a Hilbert scheme bundle over $\S$.
The condition on $I+\bar{I}$ ensures that $I$ is a generic ideal, so locally it can be written as 
\begin{equation}\label{I-expr}
I(z,\bar{z})=\langle p^n, -\bar{p}+\mu_2(z, \bar{z})p+\mu_3(z, \bar{z}) p^2...+\mu_n(z, \bar{z})p^{n-1}\rangle
\end{equation}
where $(p,\bar{p})$ are linear coordinates on $T^{*\C}\S$ induced by $(z,\bar{z})$.
The coefficients $\mu_k$ are called \textbf{higher Beltrami differentials}. A direct computation gives $\mu_k \in \Gamma(K^{1-k}\otimes \bar{K})$. The coefficient $\mu_2$ is the usual Beltrami differential. In particular for $n=2$ we get the usual complex structure.

%The punctual Hilbert scheme admits an equivalent description as a space of pairs of commuting operators, called the \emph{matrix viewpoint}. To an ideal $I$ of $\C[x,y]$ of codimension $n$, one can associate the multiplication operators by $x$ and by $y$ in the quotient $\C[x,y]/I$, denoted by $M_x$ and $M_y$. This gives a pair of commuting operators. Conversely, to two commuting operators $(A,B)$ we can associate the ideal $I(A,B)=\{P\in\C[x,y] \mid P(A,B)=0\}$. 
%The zero-fiber $\Hilb^n_0(\C^2)$ corresponds to \emph{nilpotent} commuting operators.

%From this point of view, a higher complex structure is a \emph{gauge class of a special matrix-valued 1-form $\Phi$.} Decomposing $\Phi=\Phi_1+\Phi_2$ into its $(1,0)$ and $(0,1)$-part, the pair $(\Phi_1, \Phi_2)$ consists of commuting nilpotent matrices with $\Phi_1$ principal nilpotent (which means of maximal rank $n-1$). The link to higher Beltrami differentials is given by the fact that $\Phi_2=\mu_2\Phi_1+\mu_3\Phi_1^2+...+\mu_n\Phi_1^{n-1}$.

\medskip
\noindent To define a finite-dimensional moduli space of higher complex structures, we have to define an equivalence relation. It turns out that the good notion is the following:
\begin{definition}[Def.3 in \cite{Fock-Thomas}]
A \textbf{higher diffeomorphism} of a surface $\Sigma$ is a hamiltonian diffeomorphism of $T^*\Sigma$ preserving the zero-section $\Sigma \subset T^*\Sigma$ setwise. The group of higher diffeomorphisms is denoted by $\Symp_0(T^*\Sigma)$.
\end{definition}
Symplectomorphisms act on $T^{*\C}\S$, so also on 1-forms. This is roughly how higher diffeomorphisms act on the $n$-complex structure, considered as the limit of an $n$-tuple of 1-forms. 

We then consider higher complex structures modulo higher diffeomorphisms, i.e. two structures are equivalent if one can be obtained by the other by applying a higher diffeomorphism.
Locally, any two $n$-complex structures are locally equivalent (see \cite[Theorem 1]{Fock-Thomas}).

\medskip
\noindent We define the \textbf{moduli space of higher complex structures}, denoted by $\T^n$, as the space of $n$-complex structures modulo higher diffeomorphisms.
The main properties are given in the following theorem:
\begin{thm}[Theorem 2 in \cite{Fock-Thomas}]\label{mainresultncomplex}
For a surface $\Sigma$ of genus $g\geq 2$ the moduli space $\T^n$ is a contractible manifold of complex dimension $(n^2-1)(g-1)$. 

In addition, there is a forgetful map $\mathcal{T}^n \rightarrow \mathcal{T}^{n-1}$ and a copy of Teichmüller space $\mc{T}^2\rightarrow \T^n$.
Along the copy of Teichmüller space, its cotangent space is given by the Hitchin base, i.e. for $\mu\in \mc{T}^2$ we have 
$$T^*_{\mu}\mathcal{T}^n = \bigoplus_{m=2}^{n} H^0(\Sigma,K^m).$$
\end{thm}

The forgetful map in coordinates is just given by forgetting the last Beltrami differential $\mu_n$. The copy of Teichmüller space is given by $\mu_3=...=\mu_n=0$ (this relation is unchanged under higher diffeomorphisms).

We notice the similarity to Hitchin's component, especially the contractibility, the dimension and the copy of Teichmüller space inside.

\subsection{Cotangent bundle and spectral curve}\label{cotang-spectralcurve}

The main actor in this paper is the cotangent bundle of the moduli space of higher complex structures. It can be described as follows:

\begin{thm}[Theorem 3 in \cite{Fock-Thomas}]\label{conditionC}
The cotangent bundle to the moduli space of $n$-complex structures $\cotang$ is given by 
\begin{align*}
\Big\{& (\mu_2, ..., \mu_n, t_2,...,t_n) \mid  \mu_k \in \Gamma(K^{1-k}\otimes \bar{K}), t_k \in \Gamma(K^k) \text{ and } \; \forall k\\
& (-\bar{\partial}\!+\!\mu_2\partial\!+\!k\partial\mu_2)t_{k}+\sum_{l=1}^{n-k}((l\!+\!k)\partial\mu_{l+2}+(l\!+\!1)\mu_{l+2}\partial)t_{k+l}=0 \Big\} \Big/\Symp_0(T^*\S).
\end{align*}
\end{thm}
We recognize conditions $(\mc{C})$ from Equation \eqref{condition-C-2}. We can consider it as the statement that the $t_k$ are ``holomorphic'' with respect to the higher complex structure given by the $\mu_k$. Note that we have to take equivalence classes under the action of higher diffeomorphisms.

%The infinitesimal action of $\Symp_0(T^*\S)$ on $\cotang$ is given as follows: a point in $\cotang$ is an ideal $I$ of the form \eqref{ideal-Hilb}. The two generators are polynomial functions on $T^{*\C}\S$. Hence the variation under a function $H\in\mathcal{C}^\infty(T^*\S)$ generating the higher diffeomorphism is given by the Poisson bracket of $H$ with the generators of the ideal, modulo $I$.

\medskip
Finally, Let us describe the construction of the \emph{spectral curve} associated to a point in $\cotang$. It is a submanifold of $T^{*\C}\S$ which is a branched cover over $\S$.

Define polynomials $P(p)=t_2p^{n-2}+...+t_n$ and $Q(p)=\mu_1+\mu_2 p+...+\mu_n p^{n-1}$ where $\mu_1$ is given by: $\mu_1=-\sum_{k=2}^{n-1} \frac{k}{n}t_k \mu_{k+1} \mod t^2$. Put $I= \left\langle -p^n+P(p), -\bar{p}+Q(p) \right\rangle$. Define the \textbf{spectral curve} $\tilde{\Sigma} \subset T^{*\mathbb{C}}\Sigma$ by the equations $p^n=P$ and $\bar{p}=Q$. It is a branched covering with $n$ sheets.

\begin{prop}[Proposition 5 in \cite{Fock-Thomas}]\label{spectralcurveprop}
We have $\{ -p^n+P,-\bar{p}+Q\} = 0 \mod I \mod t^2$ iff $I \in T^*\mathcal{T}^n$.
\end{prop}
The theorem states that the spectral curve is Lagrangian ``near the zero-section'' iff the ideal comes from a point in $\cotang$. This implies that the $t_k$ satisfy conditions $(\mc{C})$.

\section{Geometric description of character varieties via higher complex structures}\label{geometric-description}

In this section, we come to the main result of the paper: a formal description of the character variety in terms of the cotangent bundle of the moduli space of higher complex structures. 

In the description of the character variety from Theorem \ref{mainthm-1}, there are two problems:
\begin{itemize}
	\item We don't know how to prescribe the higher order terms of the parameters $\hat{t}_k$.
	\item Using a gauge transformation to put the $(1,0)$-part of a flat connection into Frobenius form does not take all of the gauge freedom. Hence, we can associate different collections of tensors to the same flat connection.
\end{itemize}
We will see how to overcome both problems. In particular we will see that different collections of tensors coming from the same flat connection are equivalent under higher diffeomorphisms. This explains the appearance of the moduli space of higher complex structures. Here is our main theorem:

\begin{thm}\label{mainthm-2}
There is a well-defined formal map
\begin{equation}\label{mainmap}
\omega: \cotang\to \Rep(\pi_1\S,\SL_n(\C))
\end{equation}
whose image is an open subset of the character variety. Since both manifolds have the same dimension, $\omega$ is locally a homeomorphism on the set of convergence.
\end{thm}

We can see our main result as the generalization of the case $n=2$, where we have seen that an open subset of $\Rep(\pi_1\S,\SL_2(\C))$ is described by the cotangent bundle of Teichmüller space (see end of Section \ref{description-for-n2}). It gives a formal reconstruction of a flat connection from the tensors we have extracted in Theorem \ref{mainthm-1}.

The idea of the proof is to prescribe in a canonical way the higher order terms of the $\hat{t}_k$ from a point in $\cotang$. In particular this uses the uniformisation theorem and the reparametrization freedom of $h$. Then, we can deduce the higher order terms of the $\hat\mu_k$ from the compatibility conditions of the rational WKB method. Finally, we have to check that the construction is well-defined on the equivalence classes.

\begin{proof}
Take a point $[(t_k,\mu_k)]_{2\leq k\leq n} \in \cotang$ and choose a representative $(t_k,\mu_k)_{2\leq k\leq n}$. We want to associate a connection of the form (see Equation \eqref{can-form-general}):
\begin{equation}\label{can-form-general-2}
D_h = hd+\begin{pmatrix} &&& \hat{t}_n(h)\\ 1&&&\vdots\\ & \ddots &&\hat{t}_2(h) \\ &&1&0 \end{pmatrix}+\begin{pmatrix} \hat\mu_1(h) &*&* & * \\ \hat\mu_2(h) &*&* & * \\ \vdots &*&*&*\\ \hat\mu_n(h)& *&* &* \end{pmatrix}
\end{equation}
with $\hat{\mu}_k(h) = \mu_k+\mathcal{O}(h)$ and $\hat{t}_k(h)=ht_k+\mc{O}(h^2)$. We proceed in several steps.

\medskip
\underline{Step 1:} Higher order terms of $\hat{t}_k$.
\vspace{0.2cm}

\noindent Let us explain how to obtain the higher order terms of $\hat{t}_k$. As we have already noticed before, the $\hat{t}_k$ are not tensors. Nevertheless, there is a well-known way to combine them to get tensors (for $k\geq 3$).

Denote by $e_{i,j}$ the matrix units and by $(H,E,F)$ the standard generators of $\mathfrak{sl}_2$ ($H=e_{1,1}-e_{2,2}, E=e_{1,2}, F=e_{2,1}$). Consider the unique irreducible representation $\rho_n$ of $\mathfrak{sl}_2$ of dimension $n$ such that $\rho_n(F)=\sum_{i=1}^{n-1} e_{i+1,i}:=J_-$ and $\rho_n(E):=J_+ \in \mathrm{Span}(e_{i,i+1})_{1\leq i\leq n-1}$.
Then we have the following (see \cite[Section 4.3]{bonora1992covariant} upon \cite{balog1990toda}):
\begin{lemma}
There is an upper unipotent gauge transforming the $(1,0)$-part of the connection from Equation \eqref{can-form-general-2} to 
$$h\partial+J_-+\sum_{k=1}^{n-1} N_k u_kJ_+^k$$
where $N_k$ are normalization constants satisfying $N_k\sum_{i=1}^{n+1-k}(J_+^k)_{i+k,i}=1 \,\forall\, k$.
Then, the $u_k$ are tensors of type $(k,0)$ for all $k\geq 3$. Only $u_2$ transforms like a complex projective structure (see formula \eqref{cp1-transfo}).
\end{lemma}
In \cite{bonora1992covariant}, the authors speak about a change between the ``Drinfeld--Sokolov gauge'' to the ``conformal gauge'' (or ``Dublin gauge''). For $n=3$ for example, we have $$u_3= \hat{t}_3-\frac{1}{2}h\del\hat{t}_2\;\text{ and }\; u_2 = \frac{1}{2}\hat{t}_2.$$

Now, since $u_k$ are tensors for $k\geq 3$, we can choose 
\begin{equation}\label{higher-order-terms-det}
u_k = f_k(h)t_k
\end{equation}
where $f_k(h)=h+f_k^{(2)}h^2+...$ is a formal function in $h$, which can be thought of as a reparametrization of the parameter $h$. We need these functions $f_k$ to ensure our map $\omega$ to be well-defined (see Step 3 below). We then use the inverse transformation to determine the higher order terms of $\hat{t}_k(h)$. By \cite[Equation (4.3.3a)]{bonora1992covariant}, we have $\hat{t}_k(h) = f_k(h)u_k+\text{ higher order terms}$. Thus, the lowest order term is preserved: $\hat{t}_k(h) = ht_k+\mc{O}(h^2)$.

For $u_2$, we procede as follows: a point in $\cotang$ gives in particular a point in $\T^n$. A higher complex structure induces a complex structure (which corresponds to the map $\T^n \to \T^2$). We equip our surface $\S$ with that complex structure and use the Poincaré uniformisation theorem to deduce a complex projective structure, described by the multiple of the Schwarzian derivative of some function $u$. The space of complex projective structures, which is an affine space, now becomes isomorphic to $H^0(K^2)$, holomorphic quadratic differentials. We impose
\begin{equation}\label{higher-order-t2}
\hat{t}_2(h) = f_2(h)t_2+\frac{n(n^2-1)}{12}h^2S(u,z)
\end{equation}
where $f_2(h) = h+f_2^{(2)}h^2+...$ is again a formal series and $S(u,z)$ denotes the Schwarzian derivative.

In this way, we get in a canonical way all higher order terms of $\hat{t}_k$ from a point in $\cotang$.

\medskip
\underline{Step 2:} Flat connection.
\vspace{0.2cm}

\noindent Now that we have $\hat{t}_k(h)$ to all order, we can associate a flat connection of the form \eqref{can-form-general-2}. The higher order terms of $\hat\mu_k$, denoted by $\mu_k^{(i)}$, are determined by the flatness.

Indeed, we have seen in Theorem \ref{mainthm-1} that the compatibility conditions for $s_{k/n}$ for $k<n$ are conditions $(\mathcal{C})$ and that for all $k\geq n$, we can consider the compatibility condition as an equation on the higher order of the $\hat\mu_k$. To see that, consider the first appearance of $s_{k/n}$ in system \eqref{flat-section-h-sln}. Write $k=\alpha n +\beta$, the Euclidean division. Since we know all orders of $\hat{t}_k(h)$, it is sufficient to consider the second equation of \eqref{flat-section-h-sln}:
\begin{equation}\label{step2eq}
(-h\delbar-\hat\mu_1+\hat\mu_2h\del-\hat\mu_3h^2\del^2+...+(-1)^{n}\hat\mu_n\del^{n-1})\psi = 0.
\end{equation}
With the rational WKB ansatz, at level $h^{(\alpha n+\beta)/n}$, there is the term $\mu_{\beta+1}^{(\alpha)}(\del s_{1/n})^{\beta}$, which is the first appearance of $\mu_{\beta+1}^{(\alpha)}$ in \eqref{step2eq}. Hence, the compatibility condition for $s_{k/n}$ can be seen as an equation for $\mu_{\beta+1}^{(\alpha)}$.

Finally, since the rational WKB method gives a basis of flat sections, we know that the connection is flat.

\medskip
\underline{Step 3:} Action of higher diffeomorphisms.
\vspace{0.2cm}

\noindent It remains to show that whenever we choose another representative of $[(t_k,\mu_k)]_{2\leq k\leq n}$, we get a gauge equivalent flat connection. By definition, a higher diffeomorphism admits an infinitesimal generator, a function $H$ on $T^*\S$, so it is sufficient to prove that for an infinitesimal change of $(t_k,\mu_k)$, the connection changes by an infinitesimal gauge action.

The variation of the parameters $(t_k,\mu_k)$ under $H$ is described via Poisson brackets. Put $P(p) = t_2p^{n-2}+...+t_n$, $Q(p)=\mu_1+\mu_2p+...+\mu_np^{n-1}$ and $I=\langle p^n-P(p),-\bar{p}+Q(p)\rangle$. Then the variation is described by (see Equation (2.4) in \cite{thomas2020higher})
\begin{equation}\label{var-P-Q}
\delta P = \{H,-p^n+P\} \mod I \;\text{ and }\; \delta Q = \{H,-\bar{p}+Q\} \mod I.
\end{equation}

To describe the associated gauge transformation, recall that we have a basis of the vector bundle of the form $B=(v,\nabla v,...,\nabla^{n-1}v)$. The parameters $\hat{t}_k$ and $\hat\mu_k$ are described by expressing $\nabla^nv$ and $\bar\nabla v$ in this basis (see Equations \eqref{eqqq1}).

To the Hamiltonian $H=v_1+v_2p+...+v_np^{n-1}$, we associate the differential operator $\hat{H} = v_1+v_2\nabla +...+v_n\nabla^{n-1}$. This describes an infinitesimal change $v\mapsto v+\varepsilon\delta v$ where 
\begin{equation}\label{associated-gauge}
\delta v=\frac{1}{h}\hat{H}v
\end{equation}
which induces a transformation of the whole basis $B$.
The matrix $X$ of that gauge transformation is obtained by the matrix of $A_2$, the $(0,1)$-part in the standard form \eqref{can-form-general-2}, where we replace all $\hat{\mu}_k$ by $v_k/h$. 

Let us determine the variation of $\hat{t}_k$ and $\hat\mu_k$ under that gauge transformation $X$.
Put $\hat{P}=\hat{t}_2\nabla^{n-2}+...+\hat{t}_n$ and $\hat{Q}=\hat\mu_1+\hat\mu_2\nabla+...+\hat\mu_n\nabla^{n-1}$. By definition we have $\nabla^n v= \hat{P}v$. The variation $\delta \hat{P}$ satisfies
$$\nabla^n(v+\frac{\varepsilon}{h}\hat{H}v) = (\hat{P}+\varepsilon\delta\hat{P})(v+\frac{\varepsilon}{h}\hat{H}v)$$
which gives 
\begin{equation}\label{var-hat-P}
\delta \hat{P} = \frac{1}{h}[\hat{H},-\nabla^n+\hat{P}] \mod \hat{I}
\end{equation}
where $\hat{I} = \langle \nabla^n-\hat{P},-\bar{\nabla}+\hat{Q} \rangle$ is a left ideal of differential operators. The same argument shows that
\begin{equation}\label{var-hat-Q}
\delta\hat{Q} = \frac{1}{h}[\hat{H},-\bar{\nabla}+\hat{Q}] \mod \hat{I}.
\end{equation}

We have to show that our construction in Step 1 is equivariant with respect to the action of higher diffeomorphisms and the associated gauge transform $X$.
For the lowest terms, given by our parameters $(\mu_k, t_k)_{2\leq k \leq n}$, the variation is given as the limit of Equations \eqref{var-hat-P} and \eqref{var-hat-Q} for $h\to 0$. By the general property of Poisson brackets being the semi-classical limit of the commutator, we precisely get Equations \eqref{var-P-Q}.

For the higher order terms, we claim that we can adjust, at least formally, the coefficients $f_k^{(l)}$ of the functions $f_k$ in order to make the action equivariant. To be more precise, we have to check that the variation of the $\hat{t}_k$ is given by 
$$\delta \hat{t}_2\nabla^{n-2}+...+\delta \hat{t}_n = \frac{1}{h}[v_1+v_2\nabla+...+v_n\nabla^{n-1},-\nabla^n+\hat{t}_2\nabla^{n-2}+...+\hat{t}_n] \mod \hat{I}.$$
The $h^l$-term of the coefficient of $\nabla^{n-k}$ has a contribution of $t_k\delta f_k^{(l)}$ and variations of only lower order terms, where we have put a lexicographic order on $f_k^{(l)}$: $f_k^{(l)}$ is of higher order than $f_{k'}^{(l')}$ if $l>l'$ or $l=l'$ and $k>k'$. Thus, we can take it as an equation on the coefficients of $f_k$, assuming $t_k\neq 0$.

The uniformizing complex projective structure entering the definition of $\hat{t}_2$ is equivariant with respect to diffeomorphisms of $\S$ isotopic to the identity, which are precisely the restriction of higher diffeomorphisms on $\S\subset T^*\S$.
Once the $\hat{t}_k$ transform correctly, the variation of the $\hat\mu_k$ is automatically equivariant since they are uniquely determined by the lowest terms $\mu_k$ and the flatness (via the fractional WKB method). Therefore, an infinitesimal change in the representative of the point in $\cotang$ corresponds to an infinitesimal gauge transformation of the associated flat connection. Hence, the map $\omega$ is well-defined.

To show that the image is an open subset of the character variety, note that an infinitesimal gauge preserving the Frobenius form of the $(1,0)$-part of the connection is necessarily of the form $X$ above. Indeed the existence of a basis of the form $(v,\nabla v,...,\nabla^{n-1}v)$ imposes a variation of $\delta v$ given by Equation \eqref{associated-gauge} and the rest is uniquely determined by $\delta v$. Therefore, the differential $d\omega$ is injective.

Finally, by Theorem \ref{mainresultncomplex}, the dimension of $\cotang$ equals twice the dimension of the Hitchin basis, which equals the dimension of the character variety. Hence the image is open and we get locally a homeomorphism.
\end{proof}

We believe in analogy with the theorem of Gallo--Kapovich--Marden \cite{gallo2000monodromy} that the image of $\omega$ is an open \emph{dense} subset of the character variety. Whereas $\T^n$ describes a generalization of complex structures, the space $\cotang$ describes a generalization of complex projective structures. The study of the latter is work in progress.

As mentioned in the introduction, the map $\omega: \cotang \to \Rep(\pi_1\S,\SL_n(\C))$ can be seen as an analog of the non-abelian Hodge correspondence. Since there is a natural projection map $\cotang \to \T^n$ and an inclusion $\T^n\subset\cotang$, we conjecture in analogy with the Hitchin section:
\begin{conj}
The restriction of $\omega$ to the zero-section $\T^n\subset \cotang$ has its image inside the real character variety $\Rep(\pi_1\S,\SL_n(\R))$. Hence, there is a canonical homeomorphism between $\T^n$ and the Hitchin component.
\end{conj}

Note that it is not obvious whether $\omega$ is defined for $t_k=0$ (see Step 3). It might be necessary to change $\omega$. We imposed $u_k=f_k(h)t_k$ (see Equation \eqref{higher-order-terms-det}) to determine the higher order terms of $\hat{t}_k$ ($k\geq 3$). It might be necessary to include the higher Beltrami differentials $\mu_k$ in that condition. The aim is to show that there is a gauge in which the flat connection \eqref{can-form-general-2} takes the form 
\begin{equation*}
\lambda\Phi+D_A+\lambda^{-1}\Phi^*
\end{equation*}
with $\lambda=h^{-1}$, $D_A$ a connection and $(\Phi, \Phi^*)$ a pair of 1-forms with values in the endomorphisms of $V$. This is in analogy to Hitchin's approach with Higgs bundles. 
%In fact it comes from a deeper level: for any hyperkähler manifold, the twistor description will be quadratic in the parameter $\lambda\in \C P^1$ (see \cite[Theorem 3.3]{hitchin1987hyperkahler}). 

This link to Hitchin components is still work in progress.

%\subsection{Twistor space for higher complex structures}\label{big-picture}

%\section{Link to \texorpdfstring{$W$}{W}-algebras}

\bibliographystyle{alpha}
\bibliography{ref}
%\nocite{*}

\end{document}